\documentclass[11pt,a4paper,reqno]{amsart}
\usepackage{epsfig,amssymb,latexsym,amscd,amsmath,amsthm,stmaryrd}

\usepackage[all,2cell]{xy}
\xyoption{v2} \UseAllTwocells
\usepackage{hyperref}
\theoremstyle{plain}

\newtheorem{teo}{Theorem}[section]
\newtheorem{theo}[teo]{Theorem}
\newtheorem{coro}[teo]{Corollary}
\newtheorem{lema}[teo]{Lemma}

\theoremstyle{remark}

\theoremstyle{definition}

\newtheorem{defi}[teo]{Definition}
\newtheorem{obse}[teo]{Observation}

\newtheorem{example}[teo]{Example}

\newcommand{\modGX}{{}_{(G,\/\,\Bbbk [X])}\mathcal M}
\newcommand{\modGR}{{}_{(G,R)}\,\mathcal M}

\newcommand{\modKR}{{}_{(K,R)}\,\mathcal M}
\newcommand{\modK}{{}_{K}\mathcal M}

\newcommand{\modG}{{}_{G}\mathcal M}

\newcommand{\modKGb}{{}_{(K,\,\/G)}\mathcal M}

\newcommand{\qqed}{\hfill $\Box$}

\begin{document} \title[Reductive and unipotent actions]
{Reductive and unipotent actions of affine groups} \author{Walter Ferrer Santos} \address{Facultad
de Ciencias\\Universidad de la Rep\'ublica\\ Igu\'a
4225\\11400 Montevideo\\Uruguay\\} \thanks{The authors would
like to thank Csic-UDELAR and Anii for their partial
support} \email{wrferrer@cmat.edu.uy} \author{Alvaro
Rittatore} \email{alvaro@cmat.edu.uy}
\begin{abstract}
  We present a generalized version of classical geometric invariant
  theory {\em \`a la Mumford} where we consider an affine algebraic
  group $G$ acting on a specific affine algebraic variety $X$.  We
  define the notions of \emph{linearly reductive} and of
  \emph{unipotent action} in terms of the $G$ fixed point functor 
  in the category of $(G,\Bbbk [X])$--modules. In the case that
  $X=\{\star\}$ we recuperate the concept of lineraly reductive and of
  unipotent group.  We prove in our ``\emph{relative}'' context some
  of the classical results of GIT such as: existence of quotients,
  finite generation of invariants, Kostant--Rosenlicht's theorem and
  Matsushima's criterion. We also present a partial description of the
  geometry of such \emph{linearly reductive} actions.
\end{abstract} \date{\today}
\maketitle

\section{Introduction}
\label{section:intro}
Given a fixed affine algebraic group $G$, classical
geometric invariant theory {\em \`a la Mumford}, is the
study of the interplay between the following three different
mathematical themes: the geometric properties of the actions
of the group $G$ on algebraic varieties, the structure of
the category $\modG$ of all rational $G$--modules and
thirdly the inner structure of the group $G$ -- or of its
Lie algebra.

Even though this theory has been extremely successful 
in the degree of generality that was originally
conceived for, it has the handicap of not being very
sensitive to the particular properties of a specific
action. 

In this paper we develop the initial features of what we call a {\em
  relative theory}.  We will be concerned with pairs $(G,X)$ where $X$
is a variety equipped with a $G$--action, and the group $G$ acts
``reductively'' in the category of $(G,\Bbbk [X])$--modules, even
though it need not be {\em reductive}. In this context (weaker than
the standardt one) we still obtain for $X$ and other varieties
naturally related to it, some of the usual results of classical
invariant theory.

The following theorem, due to E.~Cline, B.~Parshall and
L.~Scott and that appeared in \cite{kn:CPS}, conspicuously
illuminates this perspective of the theory and was a strong
motivation for the considerations that follow.

{\em Let $G$ be an affine algebraic group and $K \subseteq
  G$ a closed subgroup, then the homogeneous space $G/K$ is
  an affine variety if and only if all the exact sequences
  in the category $\modKGb$ split in $\modK$ --- in the
  language of the present paper, if and only if the action of
  $K$ on $G$ is linearly reductive.}

In paralel with the development of the concept of linearly reductive
action of $G$ on $X$ it is natural to define in the same context, the
notion of unipotent action.  For unipotent actions we generalize the
classical Kostant--Rosenlicht result on the closedness of the orbits
of $G$ in $X$.

If we apply the definitions of linearly reductive action or of
unipotent action (see Definitions \ref{defi:crucial2} and
\ref{defi:unipotentaction}) to the situation of a pair $(G,\{\star\})$
where $\{\star\}$ is the one point variety, we obtain the concept of
linearly reductive or of unipotent group.

As we are dealing with the properties of {\em pairs}
$(G,X)$, one of the main advantages of the relative point of
view is that we can formulate and prove {\em transitivity}
results as the following:

{\em Let $H \subseteq G$ be a closed subgroup of the group
$G$ and assume that the action of $G$ on $X$ is linearly
reductive, and that the action of $H$ on $G$ is also
linearly reductive, then the action of $H$ on $X$ is
linearly reductive}.

This result, when applied to the case that $X$ is the variety
consisting of a single point, yields the Matsushima's criterion in
characteristic zero, see for example \cite{kn:matsu,kn:CPS}, and in
such a way the classical Matsushima's criterion can be profitably
interpreted as a transitivity result.\footnote{We plan to develop in
  future work, the theory of \emph{geometrically reductive actions},
  this will guarantee the above result for the situation of positive
  characteristic.}

Next, we present a brief description of the contents of this
paper.

In Section \ref{section:basicdef} we establish in Definitions
\ref{defi:crucial2} the concepts of linearly reductive action of an
affine algebraic group $G$ on a commutative and rational $G$--module
algebra $R$, and in parallel on an affine algebraic variety $X$.

In Theorem \ref{theo:characterizations} we present different
characterizations of the concept of a linearly reductive
action, some of them with a cohomological
flavor and others in terms of total integrals
--compare with \cite{kn:doi1,kn:doi2,kn:doi3}-- and of relative
splittings. A version of the concept of Reynolds operator also appears
related to the linear reductivity of the action. 

In the particular case of a linearly reductive action of a unipotent
group $U$, we show that the target variety $X$ is $U$--isomorphic with
$U \times S$ where $S$ is affine and acted trivially by the group.
Hence, up to trivial factors, the only linearly reductive action of a
unipotent group is the action of the group on itself by translations,
see Lemma \ref{theo:caraclrunip} and Theorem \ref{theo:caraclrunip2}.

In Section \ref{section:grsubgr} we consider the 
case of a group $G$ and a closed subgroup $K \subseteq G$
acting by translations. We encounter the well known theorem of 
Cline, Parshall and Scott that in our language means that the action of $K$ on $G$ is linearly reductive if and only if $G/K$ is an affine variety. 

We finish the section proving
that in the same situation than above, the linear reductivity of the
action of $K$ on $G$ together with the linear reductivity of the
action of $G$ on $G/K$ are equivalent to the linear reductivity of $K$
(Theorem \ref{theo:homo}).

In Section \ref{section:transitivity} we deal with transitivity
results. In Theorem \ref{theo:matsu}, we prove the generalization of
Matsushima's criterion considered above. Moreover in the same theorem
we prove that in the case that $K \triangleright G$ is a normal subgroup
and $R$ is a $G$--module algebra, the linear reductivity of the action of
$K$ on $R$ and that of the action of $G/K$ on ${}^K\!R$, are
equivalent to that of the action of $G$ on $R$.

In Section \ref{section:levidecomposition} we use the above
mentioned transitivity results in order to prove a result
similar to -- but weaker than -- the classical Levi
decomposition of an algebraic group, and use it to give a characterization of
linearly reductive actions of a group $G$ on a variety $X$
(see Theorem \ref{theo:unipot_description}).

In Section \ref{section:invquo} we prove that the two main initial
results of classical geometric invariant theory hold for our weaker
context of linearly reductive \emph{actions}: the finite generation of
the invariants, and the existence of semi--geometric quotients, see
Theorems \ref{theo:fginvariants} and \ref{theo:liquotiens}.

In Section \ref{section:unipotent}, we define the relative
concept of unipotent action of an affine algebraic group on
an affine variety $X$. This is done in terms of the
faithfulness of the $G$--fixed point functor in the category
 $\modGX$. We prove that an action that is unipotent and
linearly reductive reduces to the action of $G$ on itself by
translations, and we also establish a generalization of
Kostant--Rosenlicht theorem: if the action of $G$ on $X$ is
unipotent, then the orbits of the action of $G$ on $X$ are
closed (see Theorem \ref{theo:genkr}).

\medskip

\noindent {\sc Notations:}

All the geometric and algebraic objects will be defined over
an algebraically closed field $\Bbbk $ and all the
algebras considered will be commutative. The geometric
actions are on the right and the algebraic actions on the
left. 

If $R \in \modG$ is a rational commutative $G$--module
algebra, then $\modGR$ is the category of rational
$(G,R)$--modules whose objects $M$ are rational $G$--modules
and $R$--modules subject to the following compatibility
condition: if $x \in G, r \in R, m \in M$, then $x\cdot
(rm)= (x\cdot r)(x\cdot m)$, with the obvious morphisms.

Given a right regular action $X \times G \rightarrow X$ of
$G$ on the affine variety $X$ (i.e. $X$ is a {\em regular
  $G$--variety} or simply a {\em $G$--variety}), then
$R=\Bbbk [X]$ is a rational $G$--module algebra and we
write ${}_{(K,X)}\mathcal M$ instead of $\modGX$.

If $M$ is a rational $G$--module, then $M \mapsto {}^G\!M
\subseteq M$ is the covariant left exact $G$--fixed part
functor $\modG \rightarrow \modG$.

Rational $G$--modules are $\Bbbk [G]$-comodules with
(right) coaction $\chi: M \rightarrow M \otimes \Bbbk [G]$ characterized by the condition: $\chi(m)=\sum m_0
\otimes m_1 \in M \otimes \Bbbk [G]$, if and only if $x
\cdot m = \sum m_0 m_1(x)$ for all $x \in G$ -- we use
Sweedler's notation.  If $G$, $R$ are as above, then $M
\in \modGR$, if and only if $\chi_M(r\cdot m)=\sum r_0\cdot
m_0 \otimes r_1m_1$.

If $G$ is an affine algebraic group $\mathcal {R}_u(G)$
denotes the unipotent radical of $G$.  In case that the
unipotent radical is trivial we say the $G$ is reductive and
it is well known that this is equivalent to be linearly
reductive in characteristic zero.

% See the following references for the definitions of
% reductive and of linearly (and geometrically) reductive
% affine algebraic group:
% \cite{kn:borelbook,kn:nosotros,kn:humph,kn:spr}.

\section{Reductive actions}
\label{section:basicdef}

The following definition of \emph{linearly reductive action}
will be central to our considerations.

\begin{defi}\label{defi:crucial2}\
\begin{enumerate}
\item Let $G$ be an affine algebraic group and $R$ a
rational $G$--module algebra. We say that the action of $G$
on $R$ is \emph{linearly reductive} if for every triple $(M,
J, \lambda)$ where $M\in \modGR$, $J \subseteq R$ is a
$G$--stable ideal and $\lambda:M \rightarrow R/J$ is a
surjective morphism of $(G,R)$--modules; there exists an
element $m \in {}^GM$, such that $\lambda(m)=1+J \in R/J$.
In the context above, if the action of $G$ on $R$ is given,
we say that $(G,R)$ is a linearly reductive pair.
\item In the case that $R= \Bbbk [X]$ and the action of
$G$ on $R$ is linearly reductive we say that the action of
$G$ on $X$ is linearly reductive and also that the pair
$(G,X)$ is linearly reductive.
\end{enumerate}
\end{defi}

\begin{defi}\label{defi:integral} Let $G$ be an affine
algebraic group and $R$ a rational $G$--module algebra. A
\emph{(left) integral} for $G$ with values in $R$ is an
$G$--equivariant linear map $\sigma:\Bbbk [G] \rightarrow
R$. The integral $\sigma$ is said to be \emph{total} if
$\sigma(1)=1$.
\end{defi}

\begin{obse}\label{obse:intinj} \
\begin{enumerate}
\item The $G$--equivariance condition $\sigma(x.f)=x.\sigma(f)$ of the
  integral $\sigma$, can be expressed in terms of the commutativity of
  the diagram below: 
\[\xymatrix{\Bbbk [G]\ar[r]^{\sigma}\ar[d]_{\Delta}&
R\ar[d]^{\chi}&\\ \Bbbk [G] \otimes \Bbbk [G]\ar[r]_-{\sigma \otimes \operatorname{id}} & R \otimes
\Bbbk [G] & \quad\text{i.e.}\quad \sum \sigma(f_1) \otimes f_2=
\sum \sigma(f)_0 \otimes \sigma(f)_1,}\] where $\Delta$ is the
comultiplication in $\Bbbk [G]$ and $\chi$ is the coaction in $R$.
\item Recall that in the situation above
there exists a total integral for $G$ with values in $R$, if
and only if $R$ is injective as an object in $\modG$ (see
for example \cite{kn:CPS,kn:doi1,kn:doi2,kn:doi3,kn:nosotros}).
\end{enumerate}
\end{obse}

\begin{theo}\label{theo:characterizations} Let $G$ an affine
algebraic group and $R$ a rational $G$--module
algebra. Then, the following conditions are equivalent:
\begin{enumerate}
\item The action of $G$ on $R$ is linearly reductive.
\item If $\varphi:M \rightarrow N$ is a surjective morphism
in the category $\modGR$, then $\varphi\bigl({}^G\!M\bigr)={}^G\!N$.
\item If $M \in \modGR$, then the rational cohomology groups
$\operatorname{H}^n(G, M)=\{0\}$ for all $n \neq 0$
\emph{(}the functors $M \mapsto \operatorname{H}^n(G,M)$ are
the derived functors of $M \mapsto {}^G\!M$, see
\cite{kn:ho,kn:jantzen}\emph{)}.

\item There exists a total integral $\sigma:\Bbbk [G]
\rightarrow R$.

\item The $G$--module algebra $R$ is an injective object in
the category $\modG$.

\item Every object $M \in \modGR$ is injective in $\modG$.

\end{enumerate}
\end{theo}

\begin{proof} First we prove that the first two conditions are
  equivalent. It is clear that the second condition implies the first.
  Conversely, let $M,N,\varphi\,\, \text{and}\,\,n \in {}^GN$ be as in
  (2), consider $m \in M$ such that $\varphi(m)=n$ and call
  $\overline{M}=\{\sum_{i} t_i (x_i \cdot m): t_i \in R, x_i \in G\}$
  the $(G,R)$--submodule of $M$ generated by $m$. Clearly the ideal
  $J=\{r \in R : rn=0\} \subseteq R$ is $G$--stable and we can define
  the map $\lambda: \overline{M} \rightarrow R/J$ by the rule
  $\lambda(\overline{m})= \sum_i t_i + J$. Notice that this definition
  makes sense: if we have other representation $\overline{m}=\sum_j
  u_j (y_j\cdot m)$, then applying $\varphi$ we obtain that
  $\bigl(\sum_i t_i\bigr)n=\bigl(\sum_j u_j\bigr)n$ and then $\sum_i
  t_i - \sum_j u_j \in J$. It is also clear that $\lambda$ is a
  morphism in the category $\modGR$. Using our hypothesis we find
  $\overline{m}_0 \in {}^GM$, such that $\lambda(\overline{m}_0)=
  1+J$. Then, writing $\overline{m}_0=\sum_j t_i (x_i\cdot m)$, we
  have that $\sum_i t_i \equiv 1 (\text{mod}\,J)$
  i.e. $\varphi(\overline{m}_0)=(\sum_i t_i)n= n$. The equivalence of
  (2) and (3) is clear. The equivalence of (4) and (5) was already
  mentioned in part (2) of Observation \ref{obse:intinj}, and it is
  obvious that (6) implies (5).  The proof that (4)
  implies (2) goes as follows. For an arbitrary object $M \in \modGR$
  we can consider the map $p_M:M \rightarrow M$ defined as:
\begin{equation*} p_M(m)= \sum \sigma(\mathcal Sm_1)\cdot
m_0,
\end{equation*} where $\mathcal S:\Bbbk [G]\to \Bbbk [G]$ is the antipode map $\mathcal S(f)(k)=f(k^{-1})$,
$k\in G$. By direct computation one can prove that : if $m
\in {}^G\!M$, then $p_M(m)=m$; $p_M(M)={}^G\!M$ and for an
arbitrary morphism $\varphi:M \rightarrow N$ in $\modGR$ the
diagram below commutes.

\[\xymatrix{M \ar[d]_{p_M} \ar[r]^-{\varphi}& N
\ar[d]^{p_N}\\ {}^G\!M
\ar[r]_{\varphi|_{_{{}^G\!M}}}&{}^G\!N.}\]
 
Then it is clear that if $\varphi$ is surjective, so is
$\varphi|_{_{{}^G\!M}}$.

We finish by proving that (1) implies (6).  Assume that $V
\subseteq W$ is an inclusion in the category $\modG$, $M \in
\modGR$ and let $\varphi:V \rightarrow M$ be a morphism of
$G$--modules. We want to show that it can be extended to a
morphism of $G$--modules, $W \rightarrow M$.  Assume that
$W$ is finite dimensional. Consider
$\operatorname{Hom}_{\,\Bbbk }(W,M)$ and
$\operatorname{Hom}_{\,\Bbbk }(V,M)$ as objects in
$\modGR$ in the following manner: if $x \in G$, $r \in R$
and $\psi \in \operatorname{Hom}_{\,\Bbbk }(W,M)$, then
$(x\cdot \psi)(w)=x\cdot \psi(x^{-1}\cdot w)$ and
$(r\psi)(w)=r\psi(w)$ for $w \in W$, and similarly for
$\operatorname{Hom}_{\,\Bbbk }(V,M)$. The map $\Gamma:
\operatorname{Hom}_{\,\Bbbk }(W,M) \rightarrow
\operatorname{Hom}_{\,\Bbbk }(V,M)$ given by the
restriction of the functions from $W$ to $V$ is clearly a
surjective morphism in $\modGR$. Then, by hypothesis, for
some $\widehat{\varphi} \in
{}^G\!\operatorname{Hom}_{\,\Bbbk }(W,M)$ we have:
$\Gamma(\widehat{\varphi}) = \operatorname{\varphi} \in
{}^G\!\operatorname{Hom}_{\,\Bbbk }(V,M)$ and hence the
map $\varphi$ can be extended from $V$ to $W$.

In the case that $W$ is not necessarily finite dimensional,
using standard arguments we may assume that we are in the
situation that we have extended the map $\varphi:V
\rightarrow R$ maximally to $V_\infty$ as shown in the
diagram below.
\[\xymatrix{V \ar[d]^(.4){\varphi} \ar[r]^-{\subseteq}&
{V_\infty}
\ar[dl]^(.32){\varphi_{\infty}}\ar[r]^{\subseteq}& W\\
M&&}\] If $V_\infty \neq W$, take $w\in W \setminus
V_\infty$ and consider $\langle G\cdot w\rangle$, the finite
dimensional $G$--module generated by $w$. As we just proved,
we can extend the restriction of $\varphi_{\infty}$ from
$V_{\infty} \cap \langle G\cdot w\rangle$ to $\langle G\cdot
w\rangle$. Putting together this extension of $\varphi$ with
the compatible extension $\varphi_{\infty}$ we construct an
extension of $\varphi$ to $V_{\infty} + \langle G\cdot
w\rangle$, and this is clearly a contradiction.
\end{proof}
\begin{obse}
It can be easily proved that the three conditions below are also 
equivalent the linear reductivity of the action of $G$ on $R$.
\begin{enumerate}
\item If $\theta:V \rightarrow W$ is a surjective morphism
in the category $\modG$, then $(\operatorname{id} \otimes
\theta)\bigl({}^G\!(R \otimes V)\bigr)= {}^G\!(R \otimes
W)$.
\item If $V \!\in\! \modG$, then $\operatorname{H}^n(G,
R\otimes V)=\{0\}$ for all $n \neq 0$.
\item Every inclusion $N \subseteq M$ in the category
$\modGR$ splits in $\modG$.
\end{enumerate}
\end{obse}
\begin{obse} \label{obse:propertiesreynolds}
\begin{enumerate}
\item A natural family of morphisms $\{p_M\mathrel{:} M \in
\modGR\}$ as the one constructed in the above proof is
called a family of \emph{Reynolds operators} for the action
of $G$ on $R$. It can be proved that the existence of a
family of Reynolds operators is equivalent to the linear
reductivity of the action.
\item Assume that $S \in \modGR$ is a $(G,R)$--module
algebra. Take $s_0 \in {}^G\!S$ and consider the morphism
$\ell_{s_{0}}:S \rightarrow S$ defined as
$\ell_{s_{0}}(s)=s_0s$ for $s \in S$. It is clear that
$\ell_{s_{0}}$ is a morphism in the category $\modGR$. Using
the naturality of the family of Reynolds operators we deduce
that $p_S(s_0s)=s_0p_S(s)$, equality that is valid for all
$s_0 \in {}^G\!S$ and $s \in S$. This equality can also be
formulated as follows: for all $s,t \in S$:
$p_S\bigl(p_S(s)t\bigr)=p_S(s)p_S(t)$.  In the literature
the above equality is called the \emph{Reynolds identity}.
\item \label{item:forlater}Assume now that $R$ is a rational
$G$--module algebra, and that $K \triangleleft G$ is a
closed normal subgroup of $G$. Assume that we have a family
of Reynolds operators for the category of $(K,R)$-modules. Fix $g
\in G$ and consider the map $c_g: R \rightarrow R$,
$c_g(r)=gr$. If we call $R_g$ the algebra $R$ equipped with
the action $h\cdot_g r=ghg^{-1}r$, then from the commutativity of
the diagram:
\[
\xymatrix{R\ar[r]^{p_{_R}}\ar[d]_{c_g}&{}^K\!R\ar[d]^{c_g}\\
R_g \ar[r]_{p_{_R}}&{}^K\!R }\] we deduce that
$p_R(gr)=gp_R(r)$.
\end{enumerate}
\end{obse}

\begin{example}\label{example:inicial}\
\begin{enumerate}
\item If $G$ is a linearly reductive group
-- in the classical sense -- then any rational action of $G$
on an algebra $R$ is linearly reductive. Conversely, in the
case of a trivial action of a group $G$ on an algebra $R$,
the linear reductivity of the action implies the
linear reductivity of the group. If we have a surjective
morphism in $\modG$, $\varphi:M \rightarrow N$, then
extending scalars we obtain the surjective morphism
$\operatorname{id} \otimes \varphi:R \otimes M \rightarrow R
\otimes N$ in $\modGR$. As $ (\operatorname{id} \otimes
\varphi)({}^G(R \otimes M))= (\operatorname{id} \otimes
\varphi)(R \otimes {}^GM)=R \otimes \varphi({}^GM)= R
\otimes {}^GN$, we conclude that $\varphi({}^GM)= {}^GN$.
 
\item \label{item:goingup} Using a change of scalars argument,
and Theorem \ref{theo:characterizations} it follows 
that if $\varphi:R \rightarrow S$ a morphism of rational
$G$--module algebras and the action of $G$ on $R$ is
linearly reductive, then the action of $G$ on $S$ is
linearly reductive.  In particular, if $X$ and $Y$ are
affine $G$--varieties and $f:X \rightarrow Y$ is an
equivariant morphism, then if $(G, Y)$ is linearly
reductive, so is $(G,X)$.

\item \label{item:factor}Let  $S= R \otimes L$,
where $R,L$ are $G$--module algebras. Then, from the
linear reductivity of the action of $G$ on $R$ we deduce the
same reductivity for the action on $S$. Similarly if $X,Y$
are affine $G$--varieties, and the action of $G$ on $Y$ is
linearly reductive, the same is true for $X\times Y$.

\item Let $K \triangleleft\, G$ be a (normal) inclusion of affine
  algebraic groups. Assume that $R$ is a rational $G$--module algebra
  where $K$ acts trivially. Then the action of $G$ in $R$ is linearly
  reductive if and only if the action of $G/K$ in $R$ is linearly
  reductive.

\item \label{item:hopf}The action of $G$ in itself by translations is linearly
  reductive. Indeed, the fundamental theorem on Hopf modules guarantees the
  existence of a natural equivalence in the category of $\bigl(G,\Bbbk
  [G]\bigr)$--modules: $\Theta_M:M \cong \Bbbk [G] \otimes {}^G\!M$,
  where $M$ is taken with its action and on the tensor product we use
  the diagonal action, see for example \cite[Theorem
  4.3.29]{kn:nosotros}.  This implies that a surjective morphism in
  the category of $\bigl(G,\Bbbk [G]\bigr)$--modules, restricts to a
  surjective morphism between the corresponding $G$--fixed parts.

\item \label{item:cross}Let $X$ be an affine variety and
$G$--variety. An \emph{equivariant cross section} is a morphism
$\varphi: X \rightarrow G$, with the property that for all
$x \in X$ and $k \in G$, $\Phi(x\cdot k)=\Phi(x)\cdot k$.
If $X$ admits an equivariant cross section, let
$S=\Phi^{-1}(1)$, and  endow $S$ with
the trivial action. Then $\Theta: X \rightarrow G
\times S$, $\Theta(x)=\left(\Phi(x),x\cdot
\Phi(x)^{-1}\right)$, is a $G$--equivariant isomorphism, with inverse
  $(k,s) \mapsto 
s\cdot k: G \times S \rightarrow X$. As the action of $G$ on
$G$ is linearly reductive, we deduce that the action of $G$
on $X$ is linearly reductive. In Theorem
\ref{theo:caraclrunip2} we show that if the group $G$ is
unipotent, all the linearly reductive actions are of the
above form.
\end{enumerate}
\end{example}

The following particular case of the situation of Example
\ref{example:inicial}, \eqref{item:goingup} is illustrative. In particular,
it shows that if the action of a non linearly reductive
affine group in a variety is linearly reductive, it cannot
have fixed points.
\begin{coro}\label{coro:goingup2} Let $G$ be an affine
  algebraic group and $\varphi:R \rightarrow \Bbbk $ a morphism of
  rational $G$--module algebras and assume that $\Bbbk $ is endowed
  with the trivial action. If the action of $G$ on $R$ is linearly
  reductive, then $G$ is linearly reductive.  In particular if the
  group $G$ acts linearly reductive on an affine variety with a fixed
  point, then $G$ is linearly reductive.
\end{coro}
\begin{proof} We use Example \ref{example:inicial}, (2) to
guarantee that the action of $G$ on $\Bbbk $ is linearly
reductive. Hence $G$ is linearly reductive.
\end{proof}

We can get some geometrical information about the closed
orbits of linearly reductive actions.

\begin{coro}\label{coro:goingup3}
\begin{enumerate}
\item Let $G$ be an affine algebraic group, and $H\subseteq
G$ a closed subgroup such that the homogeneous space $G/H$
is affine. If $G$ acts in a linearly reductive way on an
affine variety $X$ and $H$ has a fixed point on $X$, then
$H$ is linearly reductive.

\item Let $G$ be an affine algebraic group acting in a
linearly reductive way on an affine variety $X$ and suppose
that the action is separable. If $Y \subseteq X$ is an
affine orbit, then $Y$ is $G$--equivariantly isomorphic to
$G/H$ where $H$ is linearly reductive. In particular, if the
base field has characteristic zero, any affine orbit of a
linearly reductive action is of the form $G/H$ for $H$
reductive.
\end{enumerate}
\end{coro}

\begin{proof}
\begin{enumerate}
\item Using the a transitivity result that will be proved in
Section \ref{section:transitivity} --see Theorem
\ref{theo:matsu}-- we deduce that the action of $H$ on $X$
is also linearly reductive. As $H$ has a fixed point on $X$
we conclude from Corollary \ref{coro:goingup2} that $H$ is
linearly reductive.

\item The orbit $Y$ is $G$--isomorphic to $G/H$ that is an
affine variety.  As the action of $G$ on $X$ is linearly
reductive so is the action of $G$ on $G/H$. As the point $1H
\in G/H$ is $H$--fixed, we deduce that $H$ is linearly
reductive.
\end{enumerate}
\end{proof}

In accordance with a result we prove later (Theorem
\ref{theo:radical}), the consideration of reductive actions
in the case of an unipotent group is relevant to the
understanding of the reductivity of the actions.

\begin{lema}\label{theo:caraclrunip} Assume that $U$ is an
unipotent affine algebraic group and that $R$ is a rational
$U$--module algebra. 
\begin{enumerate}
\item Then, the action of $U$ on $R$ is
linearly reductive if and only if there exists a
multiplicative total integral $\sigma: \Bbbk [U]
\rightarrow R$.
\item \label{item:multirey}Moreover, in the above situation if $M \in {}_{(U,R)} \mathcal M$,
  $p_R$ and $p_M$ are the Reynolds operators associated to 
$R$ and $M$, we have that $p_M(rm)=p_R(r)p_M(m)$. 
\item \label{item:multirey2} In particular $p_R:R \rightarrow {}^G\!R$ is an algebra homomorphism.
\end{enumerate}
\end{lema}

\begin{proof} \begin{enumerate}
\item In accordance with Theorem
\ref{theo:characterizations}, the linear reductivity of the
action of $U$ on $R$ is equivalent to the existence of a
total integral $\sigma_0 : \Bbbk [U] \rightarrow R$. In
the case of a unipotent group, the existence of a total
integral implies the existence of a {\em multiplicative}
total integral as it is proved for example in \cite[Theorem
11.8.1]{kn:nosotros}.
\item Write $\chi_M(m)=\sum m_0 \otimes m_1$ and
  $\chi_R(m)=\sum r_0 \otimes r_1$, then
  $p_M(rm)=\sum \sigma(S(r_1m_1))r_0m_0=\sum \sigma(S(r_1))\sigma(S(m_1))r_0m_0=p_R(r)p_M(m)$. 
\item This result is a particular case of  \eqref{item:multirey}.
\end{enumerate}
\end{proof}

The theorem that follows is a characterization of linearly
reductive unipotent actions and its proof is related to the
results of \cite{kn:CPS}, and \cite[Theorem
11.8.2]{kn:nosotros}.
\medskip
\begin{theo}\label{theo:caraclrunip2} Assume that $U$ is an
unipotent affine algebraic group and that $X$ is an affine
$U$--variety. Then, the action of $U$ on $X$ is linearly
reductive if and only if, there exist a closed subvariety
$L$ of $X$ and an isomorphism of algebraic varieties,
$\Theta: X \rightarrow L \times U$ that is $U$--equivariant
when we endow $L \times U$ with the action given by the
right translation on $U$ and the trivial action on $L$.
\end{theo}

\begin{proof} The algebra morphism $\sigma:\Bbbk [U]
\rightarrow \Bbbk [X]$ induces a morphism of varieties
$\Phi: X \rightarrow U$ with the property that $\Phi(x \cdot
u)=\Phi(x)u$ for all $x \in X$, $u \in U$. If we call
$L=\Phi^{-1}(1)$, the existence of a morphism $\Theta$ was
shown in Example \ref{example:inicial},
\eqref{item:cross}. Conversely, in the situation of the
existence of $\Theta$, we obtain an isomorphism $\Bbbk [X] \cong \Bbbk [L] \otimes \Bbbk [U]$ and this
implies that the action of $U$ on $X$ is linearly reductive
(see Example \ref{example:inicial},\eqref{item:factor}).
\end{proof}
\medskip
\begin{obse}\label{obse:comments}\
\begin{enumerate}
\item In the algebraic version, if $U$ is a unipotent group, then a pair
  $(U,R)$ is   linearly reductive if and only if  there exists a decomposition
$R=\Bbbk [U] \otimes L$ for some algebra $L$ where $U$ acts trivialy.
\item Notice that in the above situation all the orbits of
the action are isomorphic with $U$ and that the quotient
variety $X/U$ exists and coincides with $L$.
\item Moreover, the above Theorem \ref{theo:caraclrunip2} means that
  except for trivial factors the action of the group $U$ on itself by
  right translations -- that we have seen it is linearly reductive in
  Example \ref{example:inicial}, \eqref{item:hopf}--
  is the only linearly reductive action of a unipotent group.
\end{enumerate}
\end{obse}

\section{The case of a group and a closed subgroup}
\label{section:grsubgr}

We consider the particular case of a group and a closed
subgroup acting by translations.

The {\em restriction} functor $\operatorname{Res} _G^K:
\modG \rightarrow \modK$ --$\operatorname{Res}
_G^K(M):=M|_K$-- has a right adjoint: the {\em induction}
functor $\operatorname{Ind}_K^G: \modK \rightarrow \modG$ .

Explicitly, if $M \in \modK$ we endow $\Bbbk [G] \otimes M \in \modK$
with the diagonal $K$--module structure, where the action on the--$G$
polynomials is given by $(x \cdot f)(g)=f(gx)$ for all $x \in K$, $f
\in \Bbbk [G]$. Then $\operatorname{Ind}_K^G(M)={}^K\!\bigl(\Bbbk [G]
\otimes M\bigr)$, with the $G$--action defined as: $g \star (\sum f_i
\otimes m_i) = \sum_i f_i\cdot g^{-1} \otimes m_i$ for $g \in G$, $f_i
\in \Bbbk [G]$ and $m_i \in M$ where $(f \cdot g^{-1})(g') =
f(g^{-1}g')$ for $g' \in G$ (see \cite{kn:CPS} or \cite[Section
6.6]{kn:nosotros}).

Moreover, if $\alpha:M_1 \rightarrow M_1 \in \modK$, then
$\operatorname{Ind}_K^G(\alpha)$ is the restriction of
$\operatorname{id} \otimes \alpha$.  

\medskip
We list some of the basic properties of the induction
functor that will be later used.

\begin{obse} \label{obse:firstproperties} In the above
context we have:
\begin{enumerate}
\item ${}^G\bigl(\operatorname{Ind}_K^G(M)\bigr)={}^K\!\bigl(\Bbbk
[G]^G \otimes M\bigr)={}^KM$.
\item The counit of the above adjunction is $\varepsilon_M:
{}^K\!\bigl(\Bbbk [G] \otimes M\bigr)\rightarrow M$;
$\varepsilon_M\bigl(\sum f_i \otimes m_i\bigr):= \sum
f_i(1)m_i$. Observe that if $\sum f_i \otimes m_i \in
{}^K\!\bigl(\Bbbk [G] \otimes M\bigr)$ then: $\sum_i
f_{i}(1) y\cdot m_{i}= \sum_i f_{i}(y^{-1})m_{i},\,\text{for
all} \, y \in K$.

\item The \emph{tensor identity} guarantees that for $M \in
\modK$ and $N \in \modG$, there is a natural isomorphism
between $\operatorname{Ind}_K^G\bigl(M \otimes N|_K\bigr) \cong
\operatorname{Ind}_K^G(M) \otimes N$. In particular if $M= \Bbbk $ and $N$ is an arbitrary rational $G$--module we have that:
$\operatorname{Ind}_K^G\bigl(N|_K\bigr) \cong {}^K\Bbbk [G] \otimes
N$. Then ${}^KN \cong {}^G\bigl({}^K\Bbbk [G] \otimes N\bigr)$.

\item If $R$ is a $K$-module algebra, then
  $\operatorname{Ind}_K^G(R)$ is also
  $G$-module algebra. 
\end{enumerate}
\end{obse}

\begin{obse}\label{obse:obsefirstprop}
If $K \subseteq G$ is a closed inclusion of affine algebraic groups,
then $K$ is \emph{observable} in $G$ if $\varepsilon_M$ is surjective
for all $M \in \modK$, moreover the following are equivalent: 

\begin{enumerate}
\item $K$ is observable in $G$; 
\item $G/K$ is a quasi--affine variety; 
\item If $\chi: K \rightarrow \Bbbk ^*$ is a multiplicative character,
  there is a polynomial $f \in \Bbbk [G]$ such that for all $x\in K$,
  $x\cdot f=\chi(x)f$ (see \cite{kn:hmbb,kn:gross} or
  \cite[Observation 10.2.4]{kn:nosotros});
\item If $0 \neq M \in \modK$, then
$\operatorname{Ind}_K^G(M)\neq 0$. 
\item For all $0 \neq I \in
\Bbbk[G]$ stable ideal, then ${}^K\!I \neq 0$.
\end{enumerate}
The proof of most of the above assertions can be found in the
monograph \cite{kn:gross} and also in \cite[Chapter
10]{kn:nosotros}. We only mention the proof that (4) implies
(3). Assuming (4) take an arbitrary rational character $\chi$ of $K$
and call $\Bbbk_{\chi ^{-1}}$ the one dimensional representation
associated to $\chi ^{-1}$. A non zero element of
$\operatorname{Ind}_K^G(\Bbbk_{\chi ^{-1}})$ is a non zero polynomial
on $G$, that is a $K$--semi invariant with character to $\chi$.

Notice also that from (3) it follows immediately that if $K \subseteq
G$ is unipotent, then it is observable.
\end{obse}

\begin{theo}
  \label{theo:cps} Let $G$ be an affine algebraic group and
$K \subseteq G$ a closed subgroup. Then, the following
conditions are equivalent.
\begin{enumerate}
\item The right action by translation of $K$ on $G$ is linearly
  reductive.
\item The functor $\operatorname{Ind}_K^G: \modK \rightarrow \modG$ is
  exact or equivalently the homogeneous space $G/K$ is an affine
  variety.

\item The subgroup $K$ is observable in $G$, and if $ I \subseteq
  {}^K\!\Bbbk [G]$ is an ideal of ${}^K\Bbbk [G]$ such that $I \Bbbk
  [G] = \Bbbk [G]$, then $I={}^K\Bbbk [G]$.
\end{enumerate}

 \end{theo}

 \begin{proof} The equivalence of conditions (1) and (2) is a part of
   Theorem \ref{theo:characterizations}.  
The equivalence of the two conditions appearing in (2)  is due to
E.~Cline, B.~Parshall and L.~Scott (\cite{kn:CPS}, see also
\cite[Theorem 11.4.5, Theorem 
  11.6.7]{kn:nosotros} for a more elementary proof).
 We prove next that (3) is
   equivalent to the other three. First, observe that condition (2)
   implies in particular that $K$ is observable in $G$. Assuming (1)
   and taking $I$ as in (3), if we write $1=\sum_{i=1}^n f_ig_i$ with
   $f_i \in I$ and $g_i \in \Bbbk [G]$, then the map $\varphi:
   \bigoplus_i \Bbbk [G] \rightarrow \Bbbk [G]$,
   $\varphi(g_1,\dots,g_n)=\sum_{i=1}^n f_ig_i$ is surjective. Then,
   by the linear reductivity of the action of $K$ on $G$, there exist
   elements $\{\overline{g}_i \in {}^K\Bbbk [G]: i=1,\dots,n\}$ such
   that $\sum_{i=1}^n f_i\overline{g}_i=1$; therefore, $I={}^K\Bbbk
   [G]$. In order to prove the converse observe that since $K$ is
   observable in $G$ then $G/K$ is quasi--affine. Therefore, for some
   $0 \neq f \in {}^K\Bbbk [G]$, the subvariety $(G/K)_f$ is non empty
   and affine. Call $ I \subseteq {}^{K}\Bbbk [G]$ the ideal generated
   by the family $\bigl\{f \in {}^K\Bbbk [G]: (G/K)_f \,\, \text{is
     affine}\bigr\}$. If $x \in G$, then $x\cdot(G/K)_f = (G/K)_{f
     \cdot x^{-1}}$. Hence, the ideal $ I$ is stable by left
   translations by elements of $G$. Hence, $ I \Bbbk [G]$ has no
   zeroes and thus $ I \Bbbk [G]=\Bbbk [G]$. Using the hypothesis (4),
   we prove the existence a finite number of elements $f_1,\dots ,f_n
   \in {}^{K}\Bbbk [G]$, with the property that generate the unit
   ideal in ${}^{K}\Bbbk [G]$ and whose principal open subsets
   $(G/K)_{f_i}$ are affine. It is well known that in this situation
   $G/K$ is an affine variety (see \cite[Theorem, 1.4.49]{kn:nosotros}).
\end{proof}

\begin{example}
The condition that $K$ is observable in $G$ cannot be ommited in 
(4), Theorem \ref{theo:cps}. Indeed, if
$G\neq \{e\}$ is a reductive group and $K= B\subsetneq G$ is a Borel subgroup,
then ${}^B\Bbbk [G] =\Bbbk$ and $G/B\neq \{p\}$ is projective.
\end{example}

We finish this section with a characterization of the linear
reductivity of the action by translations on an affine
homogeneous space.

\begin{theo} \label{theo:homo} Let $K\subseteq G$ be a
closed inclusion of affine algebraic groups.  Then, the
action of $K$ on $G$ is linearly reductive and the action of
$G$ by translations on $G/K$ is linearly reductive if and
only $K$ is linearly reductive.
\end{theo}
\begin{proof} Let $f:M \rightarrow\!\!\!\!\!\!\!\rightarrow
  N$ be a surjective morphism of rational $K$-modules. By
  the linear reductivity of $(K,G)$, it follows that
  $\operatorname{Ind}_K^G(M)\to \operatorname{Ind}_K^G(N)$
  is a surjective morphism of $\bigr(G, {}^K\,\Bbbk 
  [G]\bigr)$-modules.  Hence, the linear reductivity of
  $(G,G/K)$ guarantees that
  ${}^G\bigl(\operatorname{Ind}_K^G(M)\bigr) \to
  {}^G\bigl(\operatorname{Ind}_K^G(N)\bigr)$ is also
  surjective and then $f|_{_{{}^K\!M}}:{}^KM \to {}^KN$ is
  surjective and $K$ is linearly reductive.

Now conversely, let $K\subseteq G$ be an inclusion of affine
algebraic groups, with $K$ linearly reductive and take
$f:M\to N$ a surjective morphism of $\bigl(G,{}^K\Bbbk 
[G]\bigr)$-modules.  Then, from the linear reductivity of
$K$, one deduces that $G/K$ is an affine variety and it
follows from the Mackey's imprimitivity theorem
(\cite[Theorem 2.7]{kn:CPSMackey}) that
$\operatorname{Ind}_K^G: {}_K\mathcal M \to
{}_{\bigr(G,{}^K\Bbbk  [G]\bigr)}\mathcal M$ is an
equivalence of categories.  Then, there exists a couple of
$K$--modules $M_0, N_0$ and a surjective morphism of
$K$-modules $g: M_0 \to N_0$ such that $f=
\operatorname{Ind}_K^G(g): \operatorname{Ind}_K^G(M_0)=M\to
\operatorname{Ind}_K^G(N_0)=N$.  By the transfer principle,
it follows that ${}^GM= {}^KM_0$, ${}^GN= {}^KN_0$ and
$f|_{_{{}^GM}}= g|_{_{{}^KM_0}}$, and the surjectivity of
$f|_{_{{}^GM}}: {}^GM \rightarrow {}^GM$ follows from the
linear reductivity of $K$.
\end{proof}

\section{Transitivity results} \label{section:transitivity}

Next we present a relative version of Matsushima's
criterion. (see \cite{kn:matsu} or \cite[Theorem
11.7.1]{kn:nosotros} for the classical version).

\begin{theo}[Generalized Matsushima's
criterion] \label{theo:matsu}
\begin{enumerate}
\item \label{item:matsu1}Let $K \subseteq G$ be a closed
inclusion of affine algebraic groups and $R$ a rational
$G$--module algebra. Suppose that the action of $K$ on $G$
by translations is linearly reductive. If the action of $G$
on $R$ is linearly reductive, then the action of $K$ on $R$
is linearly reductive.
\item \label{item:matsu2}Let $K \triangleleft~G$ be a closed
normal inclusion of affine algebraic groups and $R$ a
rational $G$--module algebra. Then, the action of $G$ on $R$
is linearly reductive if and only if the action of $K$ on
$R$ is linearly reductive and the action of $G/K$ on
${}^K\!R$ is linearly reductive.
\end{enumerate}
\end{theo}
\begin{proof}
\begin{enumerate}
\item Let $\varphi:M_1 \rightarrow M_2$ be a surjective
morphism in $\modKR$ and take an arbitrary element $m_2 \in
{}^K\!M_2$.  It follows from the reductivity hypothesis on
$(K,G)$ that $\operatorname{id} \otimes \varphi:
{}^K\!\bigl(\Bbbk [G] \otimes M_1\bigl) \rightarrow
{}^K\!\bigl(\Bbbk [G] \otimes M_2\bigr)$ is a surjective
morphism of rational $G$--modules. Hence, $\operatorname
{id} \otimes \operatorname {id} \otimes \varphi: R \otimes
{}^K\!\bigl(\Bbbk [G] \otimes M_1\bigr) \rightarrow R
\otimes {}^K\!\bigl(\Bbbk [G] \otimes M_2\bigr)$ is also
a surjective morphism in the category $\modGR$. Consider the
element $1 \otimes 1 \otimes m_2 \in R \otimes
{}^K\!\bigl(\Bbbk [G] \otimes M_2\bigr)$. As the action
of an element $x \in G$ on $\sum f_i \otimes m_{2,i} \in
{}^K\!\bigl(\Bbbk [G] \otimes M_2\bigr)$ is defined as
$x\cdot\bigl(\sum f_i \otimes m_{2,i}\bigr)= \sum f_i\cdot
x^{-1} \otimes m_{2,i}$, it is clear that the element $1
\otimes 1 \otimes m_2$ is $G$--fixed. By definition of the
linear reductivity of $G$ on $R$, we deduce that there is an
element $\xi=\sum_{ij} r_i \otimes f_{ij} \otimes m_{1,ij}
\in R \otimes {}^K\!\bigl(\Bbbk [G] \otimes M_1\bigr)$
fixed by the action of $G$, and satisfying that
\begin{equation}\label{eqn:star} \sum_{ij} r_i \otimes
f_{ij} \otimes \varphi(m_{1,ij})=1 \otimes 1 \otimes m_2.
\end{equation} We evaluate the middle term of $\xi$ at $1
\in G$ and call the resulting element $m'_1= \sum_{ij}
f_{ij}(1) r_im_{1,ij}$.  Then evaluating equation
\eqref{eqn:star} also at $1 \in G$, we obtain the equality:
$\varphi(m'_1)=m_2$. All that remains to prove is that the
element $m'_1 \in {}^K\!M_1$. As the element $\xi =
\sum_{ij} r_i \otimes f_{ij} \otimes m_{1,ij}$ is
$G$--fixed, for $y \in K$ we deduce that $\sum_{ij} y\cdot
r_i \otimes f_{ij}\cdot y^{-1} \otimes m_{1,ij} = \sum_{ij}
r_i \otimes f_{ij} \otimes m_{1,ij}$. After evaluation at 1
and multiplication we deduce that $m'_1 = \sum_{ij} r_i
f_{ij}(1)m_{1,ij}=\sum_{ij} (y \cdot r_i)
f_{ij}(y^{-1})m_{1,ij} = \sum_{ij}(y\cdot r_i) f_{ij}(1)
y\cdot m_{1,ij} = y\cdot m'_1$, where the third equality is
a consequence of the considerations at the beginning of
Section \ref{section:grsubgr}, see Observation  
\ref{obse:firstproperties}. 
\item Assume that the action of $K$ on $R$ is linearly reductive and
  the action of $G/K$ on ${}^K\!R$ is linearly reductive. Given
  $\varphi:M_1 \rightarrow M_2$ a surjective morphism in the category
  $\modGR$ and take $m_2 \in {}^G\!M_2 \subseteq {}^K\!M_2$ and using
  the linear reductivity of $(K,R)$, we prove the existence a certain
  element $m_1 \in {}^K\!M_1$ with the property that
  $\varphi(m_1)=m_2$. Call $N_1 \subseteq {}^K\!M_1$ the
  $(G/K,{}^K\!R)$--module generated by $m_1$ in ${}^K\!M_1$ and $N_2
  \subseteq {}^K\!M_2$ the $(G/K,{}^K\!R)$--module generated by $m_2$
  in ${}^K\!M_2$. We can restrict $\varphi$ to a surjective morphism
  of $(G/K,{}^K\!R)$--modules, that we continue calling $\varphi:N_1
  \rightarrow N_2$.  In this situation, using the linear reductivity
  of the action of $G/K$ on ${}^K\!R$, applied to the map $\varphi:N_1
  \rightarrow N_2$ and to the element $m_2 \in {}^{G/K}\!N_2$, we find
  an element $m_1 \in {}^{G/K}\!({}^KN_1)= {}^G\!N_1$, such that
  $\varphi(m_1)=m_2$.  Conversely, assume that $(G,R)$ is linearly
  reductive, using Theorem \ref{theo:matsu} we deduce that $(K,R)$ is
  linearly reductive. To prove the remainder part consider a
  surjective morphism of $(G/K,{}^{G/K}R)$--modules $\varphi:{M_1}
  \rightarrow {M_2}$ a and $m_2 \in {}^{G/K}M_2$. Extend
  the action to view the map $\varphi$ as living in the category of
  $G$--modules and consider the surjective morphism $\operatorname{id}
  \otimes \varphi: R \otimes M_1 \rightarrow R \otimes M_2 \in
  \modGR$, and $1 \otimes m_2 \in {}^{G}(R \otimes M_2)$. Using the
  hypothesis we deduce the existence of an element $\sum r_i \otimes
  m_{1,i} \in {}^{G}(R \otimes M_1)$ such that $\sum r_i \otimes
  \varphi(m_{1,i})= 1 \otimes m_2$. Applying the Reynolds operator
  $p_R:R \rightarrow {}^{K} R$, we deduce that $\sum
  \varphi(p_R(r_i)m_{1,i})=\sum p_R(r_i)\varphi(m_{1,i})= m_2$. The
  element $\sum p_R(r_i)m_{1,i} \in M_1$ is fixed by the
  action of $G/K$ and hence, the proof is finished (see Observation
  \ref{obse:propertiesreynolds},\eqref{item:forlater}).
\end{enumerate}
\end{proof}

%\begin{coro} \qqed
%\end{coro}

\begin{obse}\label{obse:matsulinear} \
\begin{enumerate}
\item \label{item:infres} The second assertion of Theorem
  \ref{theo:matsu} can also be proved using that for a rational $G$--module
  $V$, the \emph{Inflation--Restriction sequence} for the rational
  cohomology groups:
  \begin{equation} \label{eq:infres} 0 \rightarrow
  \operatorname{H}^1(G/K, {}^K\!(R \otimes V))
  \xrightarrow{\operatorname{Inf}} \operatorname{H}^1(G, R
  \otimes V) \xrightarrow{\operatorname{Res}}
  \operatorname{H}^1(K, R \otimes V)\
  \end{equation} 
is exact: see \cite{kn:matsuhab,kn:ho,kn:jantzen}.

  Indeed, in the hypothesis of the above theorem using Theorem
  \ref{theo:characterizations}, we conclude that the first and
  the last terms of the sequence are zero. Hence, the middle
  term is also zero and that guarantees -- by the same Theorem
  \ref{theo:characterizations} -- that the action of $G$ on $R$
  is linearly reductive.

\item \label{item:casonormal} The converse of the first part
of Theorem \ref{theo:matsu} is false, as the case $G=\Bbbk ^*\times \Bbbk $, $K=\Bbbk ^*\times \{0\}$ and
$R=\Bbbk $ shows.
 
\item In particular Theorem \ref{theo:matsu} above provides a proof of
the following well known assertion: let $K \subseteq G
\subseteq H$ be a tower of inclusions of closed affine
algebraic groups. If the homogeneous spaces $G/K$ and $H/G$
are affine, then $H/K$ is also affine.

\item In geometric terms Theorem \ref{theo:matsu} can be formulated as
  follows  (see also Theorem \ref{theo:liquotiens}).

  \noindent (1) \emph {Let $K \subseteq G$ be a closed inclusion of
    affine algebraic groups and $X$ an affine $G$--variety. If the
    action of $G$ on $X$ is linearly reductive and the homogeneous
    space $G/K$ is affine, then the action of $K$ on $X$ is linearly
    reductive. }

\noindent (2)
\emph {
Let $G$ be an affine
algebraic group and $K\triangleleft G$ a closed normal
subgroup. Let $X$ be a affine $G$--variety such that (i) the
action of $K$ on $X$ is linearly reductive; (ii) the
quotient variety $X/K$ exists and it is affine and the action of
$G/K$ on $X/K$ is also linearly reductive.  Then the action
of $G$ on $X$ is linearly reductive. Conversely, if the
action of $G$ on $X$ is linearly reductive and the quotient
variety $X/K$ exists and it is affine, then the action of $G/K$
in $X/K$ is linearly reductive.}

\item In the case that $X$ is a point, or alternatively that $R$
is the base field, the above results read as follows.

\medskip
\noindent
\emph {Let $K \triangleleft\, G$ is a normal subgroup of the
affine algebraic group $G$. Then $K$ and $G/K$ are linearly
reductive if and only if $G$ is linearly reductive.}
\medskip
\end{enumerate}
\end{obse}

As an immediate application of Theorem \ref{theo:matsu} to
the case when $R=\Bbbk $, we obtain the classical
Matsushima's criterion in characteristic zero.

\begin{coro}[Classical Matsushima's
criterion]\label{coro:matsu} Let $G$ be a reductive group
and $K \subseteq G$ be closed subgroup and assume that
$\Bbbk $ has characteristic zero.  Then $G/K$ is an
affine variety if and only if $K$ is reductive.
\end{coro}

\begin{coro} Consider a tower $K \triangleleft\, G \subseteq
L$ of closed subgroups of an affine algebraic group $L$ such
that the first inclusion is normal. Then, if the quotient
variety $L/K$ is affine and the action of $G/K$ on $L/K$ is
linearly reductive, then the homogeneous space $L/G$ is
affine and conversely. \qqed
\end{coro}

The following theorem becomes natural in the context of
considerations of transitivity.

\begin{theo}
\label{theo:espind} Let $K \subseteq G$ be an inclusion of
affine algebraic groups and assume that $K$ is exact in $G$.
Let $R$ be a $K$--module algebra and consider the
$G$--module algebra $\operatorname{Ind}_K^G(R)$. If the
action of $G$ on $\operatorname{Ind}_K^G(R)$ is linearly
reductive, so is the action of $K$ on $R$.
\end{theo}
 
\begin{proof} Assume that $\varphi:M \rightarrow N$ is a
surjective morphism in the category of $(K,R)$--modules. By
the exactness hypothesis the morphism
$\operatorname{Ind}(\varphi): \operatorname{Ind}_K^G(M)
\rightarrow \operatorname{Ind}_K^G(N)$ is also a surjective
morphism in the category of
$\bigl(G,\operatorname{Ind}_K^G(R)\bigr)$--modules. As for
any rational $K$--module $M$ the $K$--invariants of $M$ and
the $G$--invariants of $\operatorname{Ind}_k^G(M)$ are
related by the equality
${}^G\bigl(\operatorname{Ind}_K^G(M)\bigr)={}^K\!M$ the result follows directly.
\end{proof}

Let $K \subseteq G$ be a closed inclusion of affine
algebraic groups and assume that $X$ is an affine
$K$--variety.  In this context we can form the \emph{induced
variety} $X *_K G$, that is the geometric quotient of
$X\times G$ by the action $K\times (X\times G)\to X\times
G$, $a\cdot (x,g)=( x\cdot a, a^{-1}g)$. It is well known
that this quotient exists, and has a natural structure of
$G$--variety.  Applying Theorem \ref{theo:espind}, we obtain
some insight in the relationship between the reductivity of
the action of $K$ on $X$ and the action of $G$ in $G*_K X$.

\begin{theo} Let $K \subseteq G$ be an inclusion of affine
algebraic groups and assume that $K$ is exact in $G$. Let
$X$ be an affine $K$-variety. Then if $X*_KG$ is affine, and
the action of $G$ on $X*_KG$ is linearly reductive, then the
action of $K$ on $X$ is also linearly reductive.
\end{theo}

\begin{proof} Since $K$ is exact in $G$, $G/K$ is affine,
and using \cite[Lemma 3.16]{kn:oaag} it follows that $X*_KG$
is affine. The rest of the assertion is just the geometric
version of Theorem \ref{theo:espind}.
\end{proof}

From the above transitivity results (Theorem
\ref{theo:matsu},\eqref{item:matsu1}) we obtain the
following characterization of linearly reductive actions.

\begin{theo}\label{theo:radical} Let $G$ be an affine
  algebraic group and $\mathcal {R}_u(G)$ its unipotent
  radical. Assume that $X$ is an affine variety equipped
  with a regular action of $G$. Then the action of
  $G$ on $X$ is linearly reductive if and only if the action
  of $\mathcal {R}_u(G)$ on $X$ is linearly
  reductive. Similarly for the situation that $R$ is a
  rational $G$--module algebra.
\end{theo}

\begin{proof} If the action of $G$ on $X$ is linearly
reductive, it follows from Theorem
\ref{theo:matsu}, \eqref{item:matsu1} that the
action of $\mathcal {R}_u(G)$ is also linearly
reductive. Conversely, in the case that $\mathcal {R}_u(G)$
acts in a linearly reductive way on $R$, as the quotient
group $G/\mathcal {R}_u(G)$ is reductive our conclusion
follows from Theorem \ref{theo:matsu}, \eqref{item:matsu2}.
\end{proof}

\section{Levi decomposition of a linearly reductive action}
\label{section:levidecomposition}

According the Theorem \ref{theo:radical}, the linear
reductivity of the action of the affine algebraic group $G$
on the affine variety $X$ is determined by the linear
reductivity of the action of the unipotent radical $\mathcal
{R}_u(G)$. In this case, we can complement the results of
Theorem \ref{theo:caraclrunip2}, and obtain a $\mathcal
{R}_u(G)$--decomposition of the variety $X$ that is similar
to the classical decomposition of an affine algebraic group
as a semi-direct product of the unipotent radical and a
closed reductive subgroup -- the Levi decomposition.

\begin{defi}\label{defi:levifactor} 
\begin{enumerate}
\item Let $G$ be an affine
algebraic group and $X$ be an affine $G$--variety. A
\emph{Levi decomposition} for the $G$--variety $X$ is an
$\mathcal {R}_u(G)$--isomorphism $X \cong \mathcal {R}_u(G)
\times L_G(X)$ where $L_G(X)$ is an affine variety endowed
with the trivial $\mathcal {R}_u(G)$--action.
 In this situation the variety $L_G(X)$ is called
a Levi factor and it is clear that the projection $X
\rightarrow L_G(X)$ is the geometric quotient of $X$ by the
$\mathcal R_u(G)$--action.
\item
Let $G$ be an affine algebraic group and $R$ a
$G$--module algebra. A Levi decomposition of $R$ is an isomorphism
of $\mathcal R_u(G)$--module algebras $R\cong
\Bbbk \bigl[\mathcal R_u(G)\bigr]\otimes L$, where $L$ is
an algebra with trivial $\mathcal R_u(G)$--action.
\end{enumerate}
\end{defi}
\begin{theo}\label{theo:radical2} 
\begin{enumerate}
\item Let $G$ be an affine
algebraic group and $X$ is an affine $G$--variety. Then the
action of $G$ on $X$ is linearly reductive if and only if
$X$ admits a Levi decomposition.
\item Let $G$ be an affine algebraic group and $R$ a
$G$--module algebra. Then $(G,R)$ is a linearly reductive
pair if and only if $R$ admits a Levi decomposition.
\end{enumerate}
% % $X$ admits a $\mathcal {R}_u(K)$ % equivariant
% decomposition of the form $X=\mathcal {R}_u(K) \times S$ %
% where $S$ is endowed with the trivial $\mathcal
% {R}_u(K)$--action.
\end{theo}
\begin{proof} (1) Assume that the action of $G$ on $X$ is
linearly reductive, then as $\mathcal {R}_u(G)$ is a normal
subgroup, it follows from previous transitivity results -- see Theorem \ref{theo:radical}-- that the action of $\mathcal
{R}_u(G)$ on $X$ is also linearly reductive. Then, the
existence of a Levi decomposition follows directly from
Theorem \ref{theo:caraclrunip2}.

Conversely, in the case there is a decomposition of $X$ as
in Definition \ref{defi:levifactor}, from the fact that the
action of $\mathcal {R}_u(G)$ is trivial on the Levi factor,
it follows that the action of this radical is linearly
reductive on $X$ (see Theorem \ref{theo:caraclrunip2}). It
follows by transitivity (see Theorem \ref{theo:radical}) that
the action of $G$ on $X$ is also linearly reductive.

\noindent
(2) We omit this part of the proof as it is standard --see
Observation \ref{obse:comments}--.
\end{proof}

\begin{obse}\label{obse:classicallevi} In the case that we
  consider an inclusion of the form: $\mathcal
  {R}_u(G)\triangleleft G$ and the action of the radical by
  translations of $G$, the multiplicative integral
  $\sigma:\Bbbk [\mathcal {R}_u(G)] \rightarrow \Bbbk [G]$ and the associated morphism $\Phi: G \rightarrow
  \mathcal {R}_u(G)$ satisfy the following properties:
\begin{enumerate}
\item \label{item:univariance}For all $x \in G$,\, $u \in
  \mathcal {R}_u(G)$, $f \in \Bbbk [\mathcal R_u(G)]$:
  $\sigma(u \cdot f)= u\cdot \sigma(f)$,\,
  $\Phi(xu)=\Phi(x)u$.
\item \label{item:conjugation}For all $x,y \in G$,\, $f \in
  \Bbbk [\mathcal R_u(G)]$:\, $\sigma(x\cdot f \cdot
  x^{-1})= x\cdot \sigma(f) \cdot x^{-1}$,\,
  $\Phi(xyx^{-1})=x \Phi(y) x^{-1}$.
\item \label{item:consequence}For all $x \in G$,\, $u \in
  \mathcal {R}_u(G)$, $f \in \Bbbk [\mathcal R_u(G)]$:
  $\sigma(f \cdot u)= \sigma(f) \cdot u $,\,
  $\Phi(ux)=u\Phi(x)$.
\end{enumerate}

The first property is just the $\mathcal
{R}_u(G)$--equivariance of the maps $\sigma$ and $\Phi$; the
second follows in a standard manner from the uniqueness of
the integral and the third is an easy consequence of the
second.

Call $L=\{\ell \in G: \Phi(\ell)=1\}$ and endow it with the
following product: $\ell \star t=\Phi(\ell
t)^{-1}\ell t $. It follows easily from the properties
\eqref{item:univariance},\eqref{item:conjugation},\eqref{item:consequence}
that the product thus defined is associative.

Define the
product on the set $\mathcal {R}_u(G) \times L$:
\begin{equation}\label{eqn:prodlevi}
(u,\ell)\cdot_\Phi(v,s):=(u(\ell v \ell^{-1})\Phi(\ell s),\ell \star s).
\end{equation}

A direct computation shows that the product defined above is
associative and that the map:
\begin{equation*}
  (u,\ell) \mapsto u\ell : ( \mathcal {R}_u(G) \times L,\cdot_\Phi) 
\rightarrow G,
\end{equation*} 
is a group isomorphism.

Suppose that the map $\Phi: G \rightarrow \mathcal {R}_u(G)$
satisfies the following cocycle condition:
\[ \Phi(xy)=\Phi(y) y^{-1}\Phi(x) y = \Phi(y) y.\Phi(x).\]
In this situation it is clear that in $L$ the product
$\star$ coincides with the product induced by $G$ --as
$\Phi(x)=\Phi(y)=1$ implies that $\Phi(xy)=1$--, and also
that the product $\cdot_\Phi$ coincides with the usual
semidirect product on $\mathcal {R}_u(G) \times L $.  It is
also clear that $\overline{\Phi}:G/\mathcal {R}_u(G) \to L$
is an isomorphism.  

The proof that in the case of
characteristic zero, the map $\Phi$ can be taken as to
satisfy the cocycle condition, is the main content of the
classical Levi decomposition --see
\cite{kn:nosotros} for a proof or \cite{kn:humphlevi,kn:mcninchlevi} for
the cohomological viewpoint and counterexamples (in positive
characteristic)--.

% Call $\Psi: G \rightarrow \mathcal {R}_u(G)$ the map:
% \begin{equation}\label{eqn:cocyclepsi}
% \Psi(x,y)=\Phi(xy)^{-1}(\Phi(y^{-1}xy)^{-1}
%\end{equation}

\end{obse}

\begin{obse} \label{obse:generalities} We recall some basic facts on
  extensions of actions. Let $G$ be an abstract group and
  $H\,\triangleleft\, G$ a normal subgroup.

  Consider an arbitrary set $L$, and the right action of $H$ on $H
  \times L$ given by the right translation on $H$ and the trivial
  action on $L$.  An extension of this action to an action of $G$ on
  $H \times L$ (that we write as $\bigl((r,\ell),g)\bigr) \mapsto
  (r,\ell) \cdot g : (H \times L) \times G \rightarrow H \times L$),
  is the same than: a) an action $\star: L \times G \rightarrow L$
  that is trivial on $L$ ($\ell \star r = \ell, r \in H, \ell \in L$); b) a cocycle $c: L \times G
  \rightarrow H$ satisfying the following conditions:
\begin{align}
  \label{eqn:zero1} c(\ell, gg')= c(\ell \star g, g')
  g'^{-1} c(\ell, g) g' \quad;\quad 
  c(\ell, h) = h.
\end{align} for all $\ell \in L,\, h \in H,\, g,g' \in G$.

Indeed, given $\star\,\,\text{and}\,\, c$ as above, the action can be defined as $(r,
\ell) \cdot g := (c(\ell, hg), \ell \star g)$ for $\ell \in
L, h \in H, g \in G$.  Conversely, if we write
$(h,\ell)\cdot g=(1,\ell)\cdot (hg)=(1,\ell)\cdot g \cdot (g^{-1}hg)$,
it is clear that if we define $(c(\ell, g),\ell \star
g):=(1,\ell)\cdot g$, we obtain our conclusion.

Moreover, if we consider equation \eqref{eqn:zero1} for the
case that $g=h \in H$, then
for all $h\in H\,,\, g \in G$ we have:
\begin{equation}\label{eqn:zero3} 
c(\ell, hg) = c(\ell, g) g^{-1}hg. 
\end{equation}
This equality implies that for 
$\ell \in L,\, h \in H,\, g,g' \in G$:
\begin{equation}
  \label{eqn:zero4} c\big(\ell \star g, g'\big) g'^{-1} c(\ell, g)
  g'=c\big(\ell \star g, c(\ell, g)g'\big).
%= \, \ell \leftharpoonup (gg').
\end{equation}
Then, the conditions given in equation \eqref{eqn:zero1} are also
equivalent to:
\begin{align}
\label{eqn:zero4} c(\ell,gg')= c(\ell \star &g,c(\ell, g)g')
\quad;\quad
  c(\ell, h) = h
\end{align} for all $\ell \in L,\, h \in H,\, g,g' \in G$.

\end{obse}

From the above considerations we deduce the following
explicit description of the linearly reductive actions of an
affine algebraic group $G$ on the affine variety $X$.

\begin{theo} \label{theo:unipot_description} Let $G$ be an
affine algebraic group and $X$ is an affine
$G$--variety. Then the action of $G$ on $X$ is linearly
reductive if and only if $X$ is $G$--equivariantly
isomorphic with a variety of the form $\mathcal
{R}_u(G) \times L$ where $L$ is an affine variety. The action of $G$
on $\mathcal {R}_u(G) \times L$ is given in terms of a pair
$(\star,c)$, where $\star: L \times G
\rightarrow L$ is a regular right action, trivial when
restricted to $\mathcal{R}_u(G)$, and $c: L
\times G \rightarrow \mathcal {R}_u(G)$ is a morphism of
varieties satisfying the following cocycle condition:
\begin{equation} \label{eqn:zero4} c(\ell, gg') = c(\ell \star g, g') g'^{-1} c(\ell, g)
  g' = c(\ell \star g, c(\ell ,g)
  g') \,\,;\,\, c(\ell, r)
  = r
%\\label{eqn:zero5} 
%s \leftharpoonup 1 = 1 \,;\, s 
%\star 1 = s \,;&\, s \leftharpoonup z = z
\end{equation} where $\ell \in L$, $g,g' \in G$ and $r \in
\mathcal {R}_u(G)$.

In this situation the action of $G$ on $\mathcal
{R}_u(G) \times L$ is given by the formula
   \begin{align} \label{eqn:zero6} (r, \ell) \cdot g \,=
\bigl(c(\ell, rg), \ell\star g\bigr) =
\bigl(c(\ell, g)g^{-1}rg, \ell \star g \bigr),
\end{align} where $r \in \mathcal {R}_u(G)$, $\ell \in L$
and $g \in G$
\end{theo}
\begin{proof} The proof follows directly from Theorem
\ref{theo:radical2} and Observation \ref{obse:generalities}.
\end{proof} 
% % In the situation above, and in accordance with the
% considereations of % Theorem \ref{theo:caraclrunip2} there
% is a $\mathcal % {R}_u(K)$--equivariant morphism $\Phi:X
% \rightarrow \mathcal % {R}_u(K)$ that induces an isomorphism
% $\Theta:X \rightarrow \mathcal % {R}_u(K) \times S$, where
% $S=\Phi^{-1}(1) \subset X$. In explicit % terms
% $\Theta(x)=(\Phi(x), x \cdot \Phi(x)^{-1})$. Hence, the
% action % of $K$ on $X$ induces an action on the product %
% $\mathcal{R}_u(K) \times S$.  % That in turn induces a pair
% of maps that we call $\alpha: (\mathcal % {R}_u(K) \times S)
% \times K \rightarrow \mathcal {R}_u(K)$, % $\beta: (\mathcal
% % {R}_u(K) \times S) \times K \rightarrow S$ and $(u,s)\cdot
% k=(\alpha % (u,s,k), \beta(u,s,k))$.  % Explicitly:
% $\alpha(u,s,k)= \Phi (s\cdot(uk))= \Phi(s\cdot k) % k^{-1} u
% k \in \mathcal {R}_u(K)$ and $\beta(u,s,k)= s\cdot (uk) %
% \cdot \Phi(s\cdot(uk))^{-1}= s\cdot k \cdot % \Phi(sk)^{-1}
% \in S$, for $s \in S,\, u \in \mathcal {R}_u(K),\, k \in %
% K$.
  
% % A direct computation shows that in fact, the formula
% $(s,k) \mapsto % \beta(1,s,k)$ induces a % right regular
% action of $K$ on $S$ that is trivial for $k \in \mathcal %
% {R}_u(K)$. This action will be denoted as
% $\beta(1,s,k)=s\star k$.  % Then, we have that $(u,s)\cdot
% k=(\alpha % (u,s,k), s \star k)$ and the map $\alpha$
% satisfies the following % condition: for all $u \in
% \mathcal{R}_u(K),\, s \in % S,\, k,h \in K$, then
% $\alpha(\alpha(u,s,k), s \star k, h) = % \alpha(u,s,kh)$.
  
Hence, in the above perspective, a linearly reductive action
of a group $G$ on an affine variety $X$, is built from two
pieces, as in equation \eqref{eqn:zero6}. One is a classical
GIT action $\star$ of a {\em reductive} group $G/\mathcal
{R}_u(G)$ on an affine variety $L$ and the other is a rather
more intricate map defined in terms of a polynomial cocycle 
$c: L \times G \rightarrow \mathcal {R}_u(G)$
compatible with $\star$ in the sense of equation
\eqref{eqn:zero4}.

\section{Invariants and quotients}
\label{section:invquo}
Even though a group that has a lineraly reductive action on an affine
variety $X$ need not have a trivial unipotent radical, as we have
obtained a good control on the linearly reductive actions of the
unipotent groups (Theorem \ref{theo:caraclrunip2}) most of the
``classical'' results on invariants and quotients of reductive groups
can be recuperated in the relative context.

\begin{theo} \label{theo:fginvariants} Assume that $R$ is a $\Bbbk
  $--finitely generated rational commutative $G$--module algebra and
  that the action of $G$ on $R$ is linearly reductive. Let $S$ be an
  algebra object in $\modGR$ that is finitely generated over $S$, then
  ${}^G\!S$ is finitely generated over $\Bbbk $.
 \end{theo} 
 \begin{proof} We may assume that $S=R$ because in the above context,
   the action of $G$ on $S$ is linearly reductive --see Example
   \ref{example:inicial} ,\eqref{item:goingup}-- and clearly $S$ is a
   finitely generated $\Bbbk $--algebra. It follows from the
   generalized Matsushima criterion (c.f. Theorem \ref{theo:matsu}),
   that the action of $U= \mathcal R_u(G)$ on $R$ is linearly
   reductive.  Then, using the multiplicative surjective morphism
   $p_R: R \rightarrow {}^U\!R$ --see Lemma \ref{theo:caraclrunip},
   \eqref{item:multirey2}-- we conclude that the algebra of invariants
   ${}^U\!R$ is finitely generated over $\Bbbk $. Now, as ${}^G\!R =
   {}^{G/U}({}^U\!R)$, using the classical results about invariants of
   reductive groups, we conclude what we want.
\end{proof}

\begin{theo} \label{theo:lrideals}Assume that $R$ is a rational
  commutative $G$--module algebra, that the action of $G$ on $R$ is
  linearly reductive and let $S$ be an algebra object in $\modGR$.
\begin{enumerate}
\item If $I \subseteq {}^G\!S$ is an ideal, then $I= IS \cap  {}^G\!S$;
\item If $J_1,J_2 \subset S$ are $G$--stable ideals, then:
\[ J_1 \cap  {}^G\!S + J_2 \cap  {}^G\!S = (J_1 + J_2) \cap  {}^G\!S;\] 
\item If $J_1 + J_2 = S$,  then $J_1 \cap  {}^G\!S + J_2 \cap  {}^G\!S = 
{}^G\!S$. 
\end{enumerate}
 \end{theo}
\begin{proof}
\begin{enumerate}
\item Clearly $I \subseteq IS \cap {}^G\!S$. If we take
  $s=\sum_t i_ts_t \in IS \cap {}^G\!S$ with $i_t \in I, s_t
  \in S$ and apply the Reynolds operator $p_S$ we obtain
  that $s=p_S(s)=\sum i_tp_S(s_t) \in I{}^G\!S=I$.
\item It is clear that $J_1 \cap {}^G\!S + J_2 \cap {}^G\!S
  \subseteq (J_1 + J_2) \cap {}^G\!S$. If we take an element
  $s_1 + s_2 \in (J_1 + J_2) \cap {}^G\!S$, and consider
  $s_1 + J_1 \cap J_2 \in J_1/(J_1 \cap J_2)$ and $s_2 + J_2
  \cap J_2 \in J_2/(J_1 \cap J_2)$, from the equality
  $x\cdot (s_1 + s_2)=(s_1 + s_2)$ we deduce that $x\cdot
  s_1 - s_1 = s_2 - x\cdot s_2 \in J_1 \cap J_2$.  In other
  words, we have that $s_1 + J_1 \cap J_2 \in
  {}^G\!(J_1/(J_1 \cap J_2))$ and $s_2 + J_1 \cap J_2 \in
  {}^G\!(J_2/(J_1 \cap J_2))$. Using the linear reductivity
  of the action we find $\ell_1 \in {}^G\!J_1$ and $\ell_2
  \in {}^G\!J_2$ such that $\ell_1 - s_1 \in J_1 \cap J_2$
  and $\ell_2 - s_2 \in J_1 \cap J_2$. Then we write $s_1 +
  s_2 = \ell_1 + \ell_2 + t$ with $t \in {}^G\!(J_1 \cap
  J_2)$. Hence, $s_1 + s_2 \in {}^G\!J_1 + {}^G\!J_2$.
\item This condition follows immediately from the previous
  one.
\end{enumerate}
\end{proof}

The semi--geometric quotient (in particular the categorical quotient)
of a linearly reductive action exists. 

Once that Theorems \ref{theo:fginvariants} and \ref{theo:lrideals} are
established, the proof goes along the same lines than the standard
proof for actions of reductive groups on affine varieties and for that
reason we omit it (see for example \cite{kn:nosotros}[Theorem 13.2.4]
or \cite{kn:News}[Theorem 3.4]).

\begin{theo}\label{theo:liquotiens} Assume  that the regular action of an  affine
  algebraic group $G$  on the affine variety $X$ is linearly
  reductive. Consider the affine variety $Y$ having as algebra of 
  polinomial functions ${}^G\,\Bbbk [X]$ and call $\pi: X
  \rightarrow Y$ the associated morphism. Then, the pair
  $(Y,\pi)$ is a semi--geometric quotient of $X$ by $G$.\hfill\qed
\end{theo}

Hence, we have shown that if we have a linearly reductive
action of a group $G$ on an affine variety $X$, to construct
the quotient variety we can first take the quotient $X/U$,
being $U$ the unipotent radical of $G$ -- this quotient will
be affine because of Theorem \ref{theo:caraclrunip2} --; and then 
use the classical results for quotiens of affine varieties by
reductive groups in order to  obtain $X/G=\dfrac{X/U}{G/U}$.

\section{Unipotent actions}\label{section:unipotent}

Assume that $G$ is an affine algebraic group and $R$ a left
rational $G$--module algebra. In this section we define the
concept of \emph{unipotent action} in terms of the fixed
point functor: $M \mapsto {}^G\!M: \modGR \rightarrow
{}_{{}^G\!R} \mathcal M$.
Along this section we  use some of the results and definitions
recalled at the beginning of Section \ref{section:grsubgr}, particularly 
the ones in Observation \ref{obse:obsefirstprop}.

\begin{defi}
\label{defi:unipotentaction} Let $G$ and $R$ be as above.  We
say that \emph{the action of $G$ on $R$ is unipotent}, or
that \emph{the pair $(G,R)$ is unipotent}, if for all $0 \neq M \in
\modGR$, then ${}^G\!M \neq 0.$

In the case that $R=\Bbbk [X]$ for some affine variety $X$ acted
regularly by $G$ on the right, we say that the \emph{action of $G$ on
  $X$ is unipotent}, or that \emph{the pair $(G,X)$ is unipotent}, if
and only if the action of $G$ on $\Bbbk [X]$ is unipotent.
\end{defi}

It is clear that in the case that $X=\{\star\}$ or that $R=\Bbbk$, if
the action of $G$ on any of these is unipotent, then $G$ is a
unipotent group.  

Also, the action of $G$ on itself by left translations is
unipotent. Indeed, the objects $M \in {}_{(G,\Bbbk [G])} \mathcal M$
have the form $M = \Bbbk [G] \otimes {}^G\!M$ --see Example
\ref{example:inicial}--, and this implies that $M \neq 0$ if and only
if ${}^G\!M = \neq 0$ and hence the pair $(G,G)$ is unipotent.

\begin{lema}\label{lema:changeofbasis} Let $G$ be an affine
  algebraic group and $\varphi:R \rightarrow S$ a
  homomorphism of rational $G$--module algebras. Suppose the
  action of $G$ on $R$ is unipotent, then the action of $G$
  on $S$ is unipotent. Similarly, if $f:X \rightarrow Y$ is
  an equivariant $G$--morphism of affine varieties and the
  action of $G$ on $Y$ is unipotent, so is the action of $G$
  on $X$.
\end{lema}
\begin{proof} Assume that $M \in {}_{(G,S)}\mathcal M$ is a non zero
object. We may change scalars and view it as an object in
$\modGR$ that -- by hypothesis -- has non zero fixed
part. Then, by definition the action of $G$ on $S$ is
unipotent.
\end{proof}

% % Recall that in the situation of $H \subset K$ an %
% inclusion of affine algebraic groups, % the counit of the
% adjoint pair given by the {\em induction} functor %
% $\operatorname{Ind}_H^K: \modH \rightarrow \modK$ and % the
% {\em restriction} functor % $\operatorname{Res} _K^H: \modK
% \rightarrow \modH$ % is the natural transformation %
% $\varepsilon_M:{}^H(\Bbbk [K] \otimes M) \rightarrow M$,
% % $\varepsilon_M(\sum f_i \otimes m_i) = \sum f_i(1)m_i$,
% defined $M % \in \modH$. The subgroup $H$ is said to be {\em
% observable in $K$}, if % the natural transformation
% $\varepsilon$ is surjective. See for % example
% \cite{kn:nosotros,kn:fog} for a detailed study of this
% concept.

We have the following
transitivity result that yields in the particular case of a group and
a subgroup, the equivalence of the concept of observability with that
of unipotent action.  

\begin{theo}\label{theo:changeofgroup} Assume that $K \subseteq G$ is
  a closed inclusion of affine algebraic groups. 
\begin{enumerate} 
\item Let $R$ be a rational $G$--module algebra and assume that $K$ is
  observable in $G$. If
  the action of $G$ on $R$ is unipotent, the same is true for the
  action of $K$ on $R$. Similarly if the action of $G$ on an affine
  variety $X$ is unipotent, so is the action of $K$.
\item   The action of $K$ on $G$ is
  unipotent if and only if $K$ is observable in $G$.  

% so is the
%   action of $K$ on $R$. In particular if $G$ is unipotent and the
%   action of $K$ on $G$ is unipotent (in particular if
%   $G/K$ is a quasi affine variety ), then $K$ is unipotent.
\end{enumerate}
\end{theo}
\begin{proof} \begin{enumerate}
  \item If $M$ is a non--zero object in ${}_{(K,\,R)}\mathcal M$, then
    $\operatorname{Ind}_K^G(M)$ an object in the category of
    $\bigl(G,\,\operatorname{Ind}_K^G(R)\bigr)$--modules that in
    accordance with Observation \ref{obse:obsefirstprop}is not
    zero. Using the tensor identity and the fact that $R$ is a
    $G$--module --see observation \ref{obse:firstproperties}-- we
    deduce that $\operatorname{Ind}_K^G(R) = {}^K\Bbbk [G] \otimes
    R$. Hence, one can view $\operatorname{Ind}_K^G(M)$ as an
    $R$--module via the natural inclusion of $R$ into ${}^K\Bbbk [G]
    \otimes R$. In explicit terms if $\sum f_i \otimes m_i \in
    \operatorname{Ind}_K^G(M)={}^K(\Bbbk [G] \otimes M)$ and $r \in
    R$, then $r\cdot (\sum f_i \otimes m_i)=\sum S(r_1)f_i \otimes
    r_0\cdot m_i$. Considering $\operatorname{Ind}_K^G(M)$ as a object
    in $\modGR$ and using the unipotency of $G$ in $R$ we
    conclude that ${}^KM = {}^G\operatorname{Ind}_K^G(M) \neq 0$.
  \item An inclusion $K \subseteq G$ is observable if and only if for
    all $K$--stable ideals $I \subseteq \Bbbk [G]$, $I^K = 0$ implies
    that $I = 0$ (see \cite[Theorem 2.9]{kn:nosotros}). Hence the
    condition of unipotency of the action implies the observability. 
Conversely, if $K \subseteq G$ is observable, using the fact that $G$
is unipotent in $G$ and the transitivity result just proved, we
conclude that $K$ acts unipotently on $G$. 
\end{enumerate}
\end{proof}

\begin{obse}\label{obse:unipinicial}\
\begin{enumerate}
\item \label{item:unipusual}It is clear that in the case that $R$ is
  the base field $\Bbbk $, the above concept of unipotent action
  coincides with the concept of unipotent group. In particular, it
  follows directly (or from Lemma \ref{lema:changeofbasis}) that if
  $G$ is a unipotent group, then the action of $G$ on an arbitrary
  algebra $R$ is also unipotent.
\item \label{item:uniprelunip} Suppose that $G$ is an affine
algebraic group and $R$ a $G$--module algebra such that
$(G,R)$ is unipotent. Assume there is a $G$--equivariant
augmentation $\varepsilon:R \rightarrow \Bbbk $. Using
Lemma \ref{lema:changeofbasis} and the previous
consideration, we deduce that the group $G$ is unipotent. In
particular if the group $G$ acts unipotently on an affine
variety and has a fixed point, then $G$ is a unipotent
group.
\item If follows from the fact that exact subgroups are
observable and from Theorem \ref{theo:changeofgroup}, that in
the case of an inclusion $K \subseteq G$ with $K$ exact in
$G$, if a pair $(G,R)$ is unipotent, the same happens with
the pair $(K,R)$. In particular, this is true for the
situation that $K \triangleleft G$.
\end{enumerate}
\end{obse}

Next we present a relative version of Kostant--Rosenlicht
theorem concerning closed orbits of unipotent groups (see
for example
\cite{kn:rosenlicht,kn:nosotros,kn:hobook,kn:humph}
for the original result).

\begin{theo}[Generalized Kostant--Rosenlicht
theorem]\label{theo:genkr} Let $G$ be an affine algebraic
group and $X$ an affine $G$--variety. Assume that the action
of $G$ on $X$ is unipotent, then all the orbits of $G$ on
$X$ are closed. Moreover, if the action is separable, all
the orbits are of the form $G/H$ where $H$ is an
unipotent closed subgroup of $G$.
\end{theo}
\begin{proof} Assume that $O$ is an orbit of the action of
  $G$ on $X$, call $Y=\overline O$ and $C=
  \overline{O}\setminus O$. If $O \neq \overline{O}$, then
  $C$ is closed non empty and $G$--stable, $\emptyset \neq C
  \subseteq Y \subseteq X$.  By our assumption, we can find
  $f \in \Bbbk [Y]$ that is zero on $C$ and not zero in a
  point $p \in O$.  Using Lemma \ref{lema:changeofbasis} we
  may assume that $Y=X$.  Consider the non zero $G$--stable
  ideal $I$ of $\Bbbk [X]$ generated
  by $f$. A generic element $g \in I$ is of the form:
  $g=\sum g_i (x_i \cdot f)$, with $g_i \in \Bbbk
  [X]$ and $x_i \in G$.  Using the hypothesis of unipotency
  of the action of $G$ on $X$ we can guarantee the existence
  of element $0 \neq g \in {}^G\!I$. Then, $g \in \Bbbk [X]$ is constant in the orbit $O$ of $p$, and then it is
  constant in $X$, the closure of the orbit. It is clear by
  the construction of $f$ that for an arbitrary point $q \in
  C$, $g(q)=\sum g_i(q)f(q \cdot x_i)= 0$. Then, $g=0$ and
  this is a contradiction.

  As to the proof of the last assertion we proceed as
  follows.  Let $Y$ be a (closed) orbit. Hence for some $H
  \subseteq G$, the homogeneous space $G/H$ is an
  affine variety isomorphic to $Y$. It follows from Lemma
  \ref{lema:changeofbasis} and Theorem \ref{theo:changeofgroup}, that the
  action of $H$ on $G/H$ is unipotent. Since $H1$
  is a fixed point on $G/H$ for this action, we deduce that
  $H$ is a unipotent group (see Observation
  \ref{obse:unipinicial} , (\ref{item:uniprelunip})).
\end{proof}

\begin{obse} \
  \begin{enumerate}
\item If $G$ acts in a unipotent way on $X=\{\star\}$, then
  by Observation \ref{obse:unipinicial},
  (\ref{item:unipusual}) {\em all} the actions of $G$ on
  arbitrary varieties are unipotent. Then, all the orbits of
  $G$ are closed in accordance to the above generalized
  version of Kostant--Rosenlicht theorem and we recover the
  classical theorem.
\item The last assertion of Theorem \ref{theo:genkr} should
  be interpreted as a generalization of the fact that a
  closed subgroup of a unipotent group is always unipotent.
\end{enumerate}
\end{obse}

Recall the following definition that is useful in relation
to the concept of observability (for the particular
group--subgroup situation see the beginning of Section
\ref{section:grsubgr}).

\begin{defi} \label{defi:extendable} Let $G$ be an affine
  algebraic group acting regularly on an affine variety
  $X$. We say that a character $\chi:G \rightarrow \Bbbk^*$ is
  extendible (to $X$), if there is a non zero 
  $\chi$--semi--invariant polynomial in $\Bbbk [X]$. The
  multiplicative monoid of extendible characters will be
  denoted as $\mathcal{E}_G(X)$ and the group of all
  characters as $\mathcal{X}(G)$.
\end{defi}

Next lemma generalizes the fact that unipotent groups have
no characters.

\begin{lema}
  \label{lem:unipexten} Let $G$ be an affine algebraic group
  acting regularly on an affine variety $X$. If the action
  of $G$ on $X$ is unipotent, then
  $\mathcal{E}_G(X)=\mathcal{X}(G)$.
\end{lema}
\begin{proof} Let $\chi:G \rightarrow \Bbbk ^* \in
  \mathcal{X}(G)$, be a rational character. We call $\Bbbk_\chi[X]$ the $(G,\Bbbk [X])$--module consisting of
  $\Bbbk [X]$ endowed with the following ``twisted''
  $G$--module structure: if $f \in \Bbbk _\chi[X]$ and $x
  \in G$, $x \rightharpoonup_\chi f=\chi^{-1}(x) x\cdot f$
  where $(x,f) \mapsto x \cdot f$ is the original action of
  $G$ on $\Bbbk [X]$.  An element $f \in 
  {}^G\,\Bbbk_\chi[X]$ is a $\chi$--semi--invariant polynomial
  on $X$.  Hence, all characters are extendible.
\end{proof}

Next we compare the concept of unipotent action with the concept of
observable action as defined in \cite{kn:oaag}.

\begin{defi}( \cite{kn:oaag}) An action of an affine algebraic group
  $G$ on an affine variety $X$ is said to be {\em observable} if for
  every non zero $G$--stable ideal $I\subseteq \Bbbk [X]$, then
  ${}^GI\neq \{0\}$. 
\end{defi}
\begin{obse}\label{obse:unipobse}\

\begin{enumerate}
\item It follows immediately Definition
  \ref{defi:unipotentaction} that unipotent actions are observable. It
  is also clear that in the case of a group and a subgroup the action
  is observable if and only if the group $K$ is observable in $G$.
It is also clear that in this context, the concept of observable
action and of unipotent action coincide (see \ref{theo:changeofgroup}). 
\item In accordance with \cite[Corollary 3.14]{kn:oaag} an
  action of a connected affine algebraic group $G$ on a
  factorial variety $X$ is observable if and only if the
  following two conditions hold: (i) there is an open subset
  of closed orbits of maximal dimension; (ii) $\mathcal
  E_G(X)$ is a group.  

  In the presence of the stronger condition of unipotency of
  the action, the above properties are satisfied: (i) all
  the orbits are closed; (ii) $\mathcal E_G(X)$ is the group
  of all characters.

\item The action of $\Bbbk ^*$ on $\Bbbk ^2$ given by
  $(a,b)\cdot t=(at,bt^{-1})$ is not unipotent as it has non
  closed orbits but it is observable as follows from
  \cite[Theorem 4.4]{kn:oaag} --also a direct verification
  shows the validity of conditions (i) and (ii) in this
  context.
\end{enumerate}
\end{obse}

Next we show that the only actions that are at the same time
unipotent and linearly reductive, are trivial in the sense
of Theorem \ref{theo:uni+gr=trivial}.

\begin{theo}\label{theo:uni+gr=trivial} Let $G$ be an affine
  algebraic group acting in a separable way on the affine
  variety $X$. If the action of $G$ on $X$ is unipotent and
  linearly reductive, then all the orbits are closed and
  isomorphic to $G$.
\end{theo}
\begin{proof} Consider an orbit $O$ of $G$ on $X$. The
  generalized theorem of Kostant--Rosenlicht --Theorem
  \ref{theo:genkr}-- guarantees that all orbits are closed
  and using Example \ref{example:inicial},
  \eqref{item:goingup} and Lemma \ref{lema:changeofbasis} we
  deduce that the action of $G$ on $O$ is also linearly
  reductive and unipotent. By the separability hypothesis,
  we deduce that the orbits are of the form $G/G_x$ for some
  closed subgrup of $G$ and using Theorem \ref{theo:homo}
  and Theorem \ref{theo:genkr} we deduce that $G_x$ is at
  the same time linearly reductive and unipotent. Hence it
  is trivial.  
\end{proof}

\end{document}